\documentclass[11pt,oneside]{article}
\usepackage{color}
\usepackage{xcolor}
\usepackage{psfrag}
\usepackage{epsfig}
\usepackage{graphicx}
\usepackage{amsmath}
\usepackage{amssymb}
\usepackage{amsthm}
\usepackage{cite}
\usepackage[latin1]{inputenc}

\newtheorem{theorem}{Theorem}[section]

\newtheorem{proposition}[theorem]{Proposition}
\newtheorem{lemma}[theorem]{Lemma}

\newtheorem{corollary}[theorem]{Corollary}
\newtheorem{remark}[theorem]{Remark}
\newcommand{\cC}{\mathcal{C}}
\newcommand{\cH}{\mathcal{H}}
\newcommand{\cN}{\mathcal{N}}
\newcommand{\cX}{\mathcal{X}}

\newcommand{\cQ}{\mathcal{Q}}
\newcommand{\cL}{\mathcal{L}}
\newcommand{\fq}{\mathbb{F}_q}

\begin{document}

\title{On plane curves given by separated polynomials and their automorphisms}
\date{}
\author{Matteo Bonini, Maria Montanucci, Giovanni Zini}

\maketitle

\begin{abstract}
Let $\mathcal{C}$ be a plane curve defined over the algebraic closure $K$ of a prime finite field $\mathbb{F}_p$ by a separated polynomial, that is $\mathcal{C}: A(y)=B(x)$, where $A(y)$ is an additive polynomial of degree $p^n$ and the degree $m$ of $B(X)$ is coprime with $p$. Plane curves given by separated polynomials are well-known and studied in the literature. However just few informations are known on their automorphism groups. In this paper we compute the full automorphism group of $\mathcal{C}$ when $m \not\equiv 1 \pmod {p^n}$ and $B(X)$ has just one root in $K$, that is $B(X)=b_m(X+b_{m-1}/mb_m)^m$ for some $b_m,b_{m-1} \in K$. Moreover, some sufficient conditions for the automorphism group of $\mathcal{C}$ to imply that $B(X)=b_m(X+b_{m-1}/mb_m)^m$ are provided. As a byproduct, the full automorphism group of the Norm-Trace curve $\mathcal{C}: x^{(q^r-1)/(q-1)}=y^{q^{r-1}}+y^{q^{r-2}}+\ldots+y$ is computed. Finally, these results are used to construct multi point AG codes with many automorphisms.
\end{abstract}



\section{Introduction}

Deep results on automorphism groups of algebraic curves, defined over a field of characteristic zero, have been achieved after the work of Hurwitz who was the first to prove that complex curves, other than the rational and the elliptic ones, can only have a finite number of automorphisms. Afterwards, a proof of Hurwitz's result which is independent from the characteristic of the ground field was provided, increasing the interest of studying curves defined over fields of positive charactestic, as e.g. finite fields. This is especially comprehensible recalling that curves in positive characteristic may happen to have much larger $K$-automorphism group compared to their genus, as the Hurwitz bound $|G| \leq 84(g-1)$ for a $K$-automorphism $G$ of a curve of genus $g \geq 2$ fails whenever $|G|$ is divisible by the characteristic of the ground field. From previous results, see e.g \cite{henn}, we know infinite families of curves $\cC$ with $|{\rm Aut}(\cC)| \sim cg^3$ or with $|{\rm Aut}(\cC)| \sim cg^2$. Although curves with large automorphism groups may have several different features, they seems to share a common property, namely their $p$-rank is equal to zero. 
This common property and this so different situation with respect to the zero characteristic case, raises the problem of constructing and studying curves of $p$-rank zero defined over finite fields with unusual properties that a complex curve cannot have. Artin-Schreier curves and, in particular, Hermitian curves are of this type. A family of such plane curves arises from separated polynomial. It consists of curves $\cX: A(Y)-B(X)$ where $p \nmid m$ with $m = \deg B(X) \geq 2$ and $A(Y)$ is any additive separable polynomial. The main known properties of $\cC$ are extracted from the local analysis of its unique singular point $P_\infty$; see \cite{Stichorig} and Section \ref{prel}. The exposition describes the genus, the Weierstrass gap sequence at $P_\infty$ and the ramification groups of its translation automorphism group fixing $P_\infty$. 
The full $K$-automorphism group of $\cC$ fixes $P_\infty$ except in two cases, namely, when $\cC$ is the Hermitian curve $Y^{p^n} -Y -X^{p^n}+1=0$ or the curve, $Y^{p^n} + Y -X^m=0$ with $m < p^n$, and $p^n \equiv -1 \pmod m$ but now other informations are known in the literature. For $p > 2$ and $m = 2$, the latter curve is hyperelliptic. Notably for $p > 2$, these hyperelliptic curves and the Hermitian curves are the only curves whose $K$-automorphism groups have order larger than $8g^3$; see \cite{henn}. Deligne-Lusztig curves provide other examples of significant curves over finite fields, namely the DLS curves of Suzuki type and the DLR curves of Ree type. They are characterised by their genera and $K$-automorphism groups. For $p = 2$, the Hermitian curves, the DLS curves, and the hyperelliptic curves $Y^2+Y +X^{2^h}+1=0$ are the only curves with $K$-automorphism groups of order larger than $8g^3$. 

In this paper we compute the full automorphism group of $\mathcal{C}$ when $m \not\equiv 1 \pmod {p^n}$ and $B(X)$ has just one root in $K$, that is $B(X)=b_m(X+b_{m-1}/mb_m)^m$ for some $b_m,b_{m-1} \in K$. Moreover, some sufficient conditions for the automorphism group of $\mathcal{C}$ to imply that $B(X)=b_m(X+b_{m-1}/mb_m)^m$ are provided. As a byproduct, the full automorphism group of the Norm-Trace curve $\mathcal{C}: x^{(q^r-1)/(q-1)}=y^{q^{r-1}}+y^{q^{r-2}}+\ldots+y$ is computed. 

An important application of curves over finite fields is the construction of certain linear codes, called Algebraic Geometric codes (AG codes for short). The parameters of an AG code constructed from a curve $\cX$ strictly depend on the geometry of $\mathcal{X}$, and in particular on two fixed divisors on $\cX$.
The Norm-Trace curve was used in the literature to construct one-point or two-point AG codes; see \cite{BR2013,G2003,MTT2008}.
In this paper we construct multi point AG codes on the Norm-Trace curve.
Our construction starts from a divisor on $\cX$ which is invariant under the whole automorphism group of the curve; hence, our codes turns out to inherit many automorphisms.

\section{Preliminary results} \label{prel}

\subsection{Curves given by separated polynomials}
Throughout the paper, $\cC$ is a plane curve defined over the algebraic closure $K$ of a prime finite field $\mathbb{F}_{p}$ by an equation
\begin{equation}\label{EqSeparated}
A(Y)=B(X),
\end{equation}
satisfying the following conditions:
\begin{enumerate}
\item $\deg(\cC) \geq 4$;
\item $A(Y)=a_n Y^{p^n} + a_{n-1} Y^{p^{n-1}}+\ldots+a_0 Y$, $a_j \in K$, $a_0,a_n \ne 0$;
\item $B(X)=b_m X^m + b_{m-1} X^{m-1} + \ldots + b_1 X + b_0$, $b_j \in K$, $b_m \ne 0$;
\item $m \not\equiv 0 \pmod p$;
\item $n \geq 1$, $m \geq 2$.
\end{enumerate}
Note that $2$ occurs if and only if $A(Y+a)=A(Y) + A(a)$ for every $a \in K$, that is, the polynomial $A(Y)$ is additive. The basic properties of $\cC$ are collected in the following lemmas; see \cite[Section 12.1]{HKT} and \cite{Stichorig}.

 \begin{lemma} \label{lem1}
The curve $\cC$ is an irreducible plane curve with at most one singular point.
\begin{itemize}
\item[\rm (i)] If $|m-p^n| =1$, then $\cC$ is non-singular.
\item[\rm (ii)] 	\begin{itemize} \item[\rm (a)] If $m > p^n+1$, then $P_\infty=(0,0,1)$ is an $(m-p^n)$-fold point of $\cC$. \item[\rm (b)] If $p^n>m+1$, then $P_\infty=(0,1,0)$ is a $(p^n-m)$-fold point of $\cC$. \item[\rm (c)] In both cases, $P_\infty$ is the centre of only one branch of $\cC$; also, $P_\infty$ is the unique infinite point of $\cC$. \end{itemize}
\item[\rm (iii)] $\cC$ has genus $g=\frac{(p^n-1)(m-1)}{2}$;
\item[\rm (iv)] Let $K(x,y)$ with $A(y)=B(x)$ denote the function field of $\cC$. 
\begin{itemize} \item[\rm (a)] A translation $(x,y) \mapsto (x,y+a)$ preserves $\cC$ if and only if $A(a)=0$;
\item[\rm (b)] these translations form an elementary abelian group of order $p^n$, and $Aut_K(K(x,y))$ contains an elementary abelian $p$-group $G$ of order $p^n$ that fixes a unique place $\mathcal{P}_\infty$ centered at $P_\infty$ and acts transitively on the zeros of $x$;
\item[\rm (c)] the sequence of ramification groups of $G$ at $\mathcal{P}_\infty$ is
$$G=G_{\mathcal{P}_\infty}^{(1)}=G_{\mathcal{P}_\infty}^{(2)}=\ldots=G_{\mathcal{P}_\infty}^{(m)}, \quad G_{\mathcal{P}_\infty}^{(m+1)}=\{1\};$$
\item[\rm (d)] $\{\mathcal{P}_\infty\}$ is the unique short orbit of $G$, and
$$div(K(x,y) / K(x,y)^G)=(p^n-1)(m+1) \mathcal{P}_\infty;$$ 
\item[\rm (e)] $K(x,y)^G$ is rational, and $\cC$ has $p$-rank zero.
\end{itemize}
\end{itemize}
\end{lemma}

\begin{lemma} \label{lem2}
Let $M$ be a $K$-automorphism group of $\cC$, and let $M_{\mathcal{P}_\infty}=M_{\mathcal{P}_\infty}^{(1)} \rtimes H$ where $p \nmid |H|$. Then
\begin{itemize}
\item[\rm (i)] $|H|$ divides $m(p^n-1)$;
\item[\rm (ii)] $|M_{\mathcal{P}_\infty}^{(1)}| \leq p^n(m-1)^2=\frac{4p^n}{(p^n-1)^2}g^2$;
\item[\rm (iii)] $|M_{\mathcal{P}_\infty}^{(1)}|=p^n$ when $m \not\equiv 1 \pmod {p^n}$, and so $g \not\equiv 0 \pmod {p^n}$;
\item[\rm (iv)] $|M_{\mathcal{P}_\infty}^{(2)}|=p^n$ when $m \equiv 1 \pmod {p^n}$, and so $g \equiv 0 \pmod {p^n}$.
\end{itemize}
\end{lemma}

\begin{lemma} \label{lem3}
The $K$-automorphism group $Aut_K(\cC)$ fixes the place $\mathcal{P}_\infty$ except in the following two cases.
\begin{enumerate}
\item
 \begin{itemize} 
\item[\rm (a)] Up to a linear substitution on $X$ and $Y$, $\cC$ is the curve $Y^{p^n}+Y = X^m$, with $m<p^n$, $p^n \equiv -1 \pmod m$;
\item[\rm (b)] $Aut_K(\cC)$ contains a cyclic normal subgroup $C_m$ of order $m$ such that $Aut_K(\cC) / C_m \cong PGL(2,p^n)$;
\item[\rm (c)] $C_m$ fixes each of the $p^n+1$ places with the same Weierstrass semigroup as $\mathcal{P}_\infty$;
\item[\rm (d)] $Aut_K(\cC) / C_m$ acts on the set of such $p^n+1$ places as $PGL(2,p^n)$.
\end{itemize}
\item \begin{itemize}
\item[\rm (a)] Up to a linear substitution on $X$ and $Y$, $\cC$ is the Hermitian curve $\mathcal H_{p^n}:Y^{p^n}+Y = X^{p^n+1}$;
\item[\rm (b)] $Aut_K(\cC) \cong PGU(3,p^n)$;
\item[\rm (c)] $Aut_K(\cC)$ acts on the set of all places with the same Weierstrass semigroup as $\mathcal{P}_\infty$;
\item[\rm (d)] $Aut_K(\cC)$ acts on the set of such places as $PGU(3,q)$ on the Hermitian unital.
\end{itemize}
\end{enumerate}
\end{lemma}

\subsection{Algebraic Geometric codes}

We introduce in this section some basic notions on AG codes. We refer to \cite{Sti} for a detailed introduction.

Let $\cX$ be a curve of genus $g$ over $\mathbb F_q$, $\mathbb F_q(\cX)$ be the field of $\mathbb F_q$-rational functions on $\cX$, $\cX(\mathbb F_q)$ be the set of $\mathbb F_q$-rational places of $\cX$.
For an $\mathbb F_q$-rational divisor $D=\sum_{P\in\cX(\fq)}n_P P$ on $\cX$, denote by
$$ \cL(D):=\{f\in\fq(\cX)\setminus\{0\}\mid (f)+D\geq0\}\cup\{0\} $$
the Riemann-Roch space associated to $D$, whose dimension over $\fq$ is denoted by $\ell(D)$.
Consider a divisor $D=P_1+\cdots P_n$ where $P_i\in\cX(\fq)$ and $P_i\ne P_j$ for $i\ne j$, and a second $\fq$-rational divisor $G$ whose support is disjoint from the support of $D$.
The \emph{functional AG code} $C_{\cL}(D,G)$ is defined as the image of the linear evaluation map
$$\begin{array}{llll}
e_D : 	& \cL (G) &\to &\mathbb{F}_q^n\\
		& f			& \mapsto & e_D(f)=(f(P_1),f(P_2),\ldots, f(P_n))\\
\end{array}.
$$
The code $C_{\cL}(D,G)$ has length $n$, dimension $k=\ell(G)-\ell(G-D)$, and minimum distance $d\geq d^*=n-\deg(G)$; $d^*$ is called the \emph{designed minimum distance} (or Goppa minimum distance).
If $n>\deg(G)$, then $e_D$ is injective and $k=\ell(G)$.
If $\deg(G)>2g-2$, then $k=\deg(G)+1-g$.
The \emph{differential code} $C_{\Omega}(D,G)$ is defined as
$$C_{\Omega}(D,G)= \left\{ (res_{P_1}(\omega),res_{P_2}(\omega),\ldots, res_{P_n}(\omega) \mid \omega \in \Omega(G-D)\right\},$$
where  $\Omega(G-D)= \{\omega \in \Omega(\mathcal X) \mid (\omega) \geq G-D\} \cup \{0\}.$ The linear code $C_{\Omega}(D,G)$ has dimension $n-\deg(G)+g-1$ and minimum distance at least $\deg(G)-2g+2$.

Now we define the automorphism group of $C_\cL(D,G)$; see \cite{GK2,JK2006}.
Let $\mathcal{M}_{n,q}\leq{\rm GL}(n,q)$ be the subgroup of matrices having exactly one non-zero element in each row and column.
For $\gamma\in Aut(\fq)$ and $M=(m_{i,j})_{i,j}\in{\rm GL}(n,q)$, let $M^\gamma$ be the matrix $(\gamma(m_{i,j}))_{i,j}$.
Let $\mathcal{W}_{n,q}$ be the semidirect product $\mathcal M_{n,q}\rtimes Aut(\fq)$ with multiplication $M_1\gamma_1\cdot M_2\gamma_2:= M_1M_2^\gamma\cdot\gamma_1\gamma_2$.
The \emph{automorphism group} $Aut(C_\cL(D,G))$ of $C_\cL(D,G)$ is the subgroup of $\mathcal{W}_{n,q}$ preserving $C_\cL(D,G)$, that is,
$$ M\gamma(x_1,\ldots,x_n):=((x_1,\ldots,x_n)\cdot M)^\gamma \in C_\cL(D,G) \;\;\textrm{for any}\;\; (x_1,\ldots,x_n)\in C_\cL(D,G). $$
Let $Aut_{\fq}(\cX)$ be the $\fq$-automorphism group of $\cX$ and
$$ Aut_{\fq,D,G}(\cX):=\{ \sigma\in Aut_{\fq}(\cX)\,\mid\, \sigma(D)=D,\,\sigma(G)\approx_D G \}, $$
where $G'\approx_D G$ if and only if there exists $u\in\fq(\cX)$ such that $G'-G=(u)$ and $u(P_i)=1$ for $i=1,\ldots,n$; note that $\sigma(G)=G$ implies $\sigma(G)\approx_D G$.
Then the following holds.
\begin{proposition}{\rm (\cite[Proposition 2.3]{BMZ2017})}\label{tivoglioiniettivo}
If any non-trivial element of $Aut_{\fq}(\cX)$ fixes less than $n$ $\fq$-rational places of $\cX$, then $Aut(C_{\cL}(D,G))$ contains a subgroup isomorphic to 
$$ ({\rm Aut}_{\fq,D,G}(\cX)\rtimes{\rm Aut}(\fq))\rtimes \mathbb{F}_q^*. $$
\end{proposition}

In the construction of AG codes, the condition ${\rm supp}(D) \cap {\rm supp}(G)=\emptyset$ can be removed as follows; see \cite[Sec. 3.1.1]{TV}.
Let $P_1,\ldots,P_n$ be distinct $\fq$-rational places of $\cX$ and $D=P_1+\ldots +P_n$, $G=\sum n_P P$ be $\fq$-rational divisors of $\cX$.
For any $i=1,\ldots,n$ let $t_i$ be a local parameter at $P_i$. The map
$$\begin{array}{llll}
e^{\prime}_{D} : 	& \cL (G) &\to &\mathbb{F}_q^n\\
		& f			& \mapsto & e^\prime_{D}(f)=((t^{n_{P_1}}f)(P_1),(t^{n_{P_2}}f)(P_2),\ldots, (t^{n_{P_n}}f)(P_n))\\
\end{array}
$$	
is linear. We define the \emph{extended AG code} $C_{ext}(D,G):=e^{\prime}(\cL(G))$.
Note that $e^\prime_D$ is not well-defined since it depends on the choise of the local parameters; yet, different choices yield extended AG codes which are equivalent.
The code $C_{ext}$ is a lengthening of $C_{\cL}(\hat D,G)$, where $\hat D = \sum_{P_i\,:\,n_{P_i}=0}P_i$.
The extended code $C_{ext}$ is an $[n,k,d]_q$-code for which the following properties still hold:
\begin{itemize}
\item $d\geq d^*:=n-\deg(G)$.
\item $k=\ell(G)-\ell(G-D)$.
\item If $n>\deg(G)$, then $k=\ell(G)$; if $n>\deg(G)>2g-2$, then $k=\deg(G)+1-g$.
\end{itemize}

\section{On the automorphism group of $\cC$}\label{Sec:Aut}

At first we consider the norm-trace curve $\mathcal{N}_{q,r}$ with affine equation
$$ X^{\frac{q^r-1}{q-1}} = Y^{q^{r-1}}+Y^{q^{r-2}}+\cdots+Y, $$
where $q$ is a $p$-power and $r$ is a positive integer.
For $r=2$, this is the $\mathbb F_{q^2}$-maximal Hermitian curve, with automorphism group isomorphic to $PGU(3,q)$.
For $r>2$, we determine the automorphism group of $\mathcal N_{q,r}$.

\begin{theorem}\label{AutNormTrace}
For $r\geq3$, $Aut_K(\mathcal N_{q,r})$ has order $q^{r-1}(q^r-1)$ and is a semidirect product $G\rtimes C$, where
$$ G=\left\{ (x,y)\mapsto(x,y+a)\mid Tr_{q^r\mid q}(a)=0 \right\}, \quad C=\{(x,y)\mapsto(b x,b^{\frac{q^r-1}{q-1}}y)\mid b\in\mathbb F_{q^r}^*\}. $$
\end{theorem}

\begin{proof}
Suppose that $\mathcal N_{q,r}\cong\cH_{\bar q}$ for some $p$-power $\bar q$.
From Lemma \ref{lem1} (iii), $g(\cN_{q,r})=g(\cH_{\bar q})$ reads $\frac{(\frac{q^r-1}{q-1}-1)(q^{r-1}-1)}{2} = \frac{\bar q(\bar q-1)}{2}$. This implies $\bar q=q$ and $r=2$, a contradiction to the assumption on $r$.

Now suppose that $\cN_{q,r}$ is isomorphic to the curve $\cX:X^s=Y^{\bar q}+Y$ for some $p$-power $\bar q$, with $s<\bar q$, $s\mid(\bar q+1)$.
From Lemma \ref{lem2}(iii), the Sylow $p$-subgroups $Aut_K(\cN_{q,r})_{\mathcal P_\infty}^{(1)}$ and $Aut_K(\cX)_{\mathcal P_\infty}^{(1)}$ of $Aut_K(\cN_{q,r})_{\mathcal P_\infty}$ and $Aut_K(\cX)_{\mathcal P_\infty}$ have order $q^{r-1}$ and $\bar q$, respectively.
From Lemma \ref{lem1}(e) $\cN_{q,r}$ and $\cX$ have zero $p$-rank. Hence, $Aut_K(\cN_{q,r})_{\mathcal P_\infty}^{(1)}$ and $Aut_K(\cX)_{\mathcal P_\infty}^{(1)}$ are Sylow $p$-subgroups of $Aut_K(\cN_{q,r})\cong Aut_K(\cX)$; see \cite[Lemma 11.129]{HKT}. Therefore $q^{r-1}=\bar q$.
Then $g(\cN_{q,r})=g(\cX)$ yields $s=\frac{q^r-1}{q-1}=\bar q+\cdots+q+1$, a contradiction to $s<\bar q$.

From Lemma \ref{lem3}, this proves that $Aut_K(\mathcal N_{q^r\mid q})$ fixes $\mathcal P_\infty$.
By direct checking $Aut_K(\cN_{q,r})$ contains the group $G\rtimes C$ defined in the statement of the theorem.
From Lemma \ref{lem3}, $Aut_K(\cN_{q,r})=G\rtimes H$, where $H$ is a cyclic group. From Schur-Zassenhaus theorem, $H$ contains $C$ up to conjugation.
By Lemma \ref{lem1}(e) the quotient curve $\cN_{q,r}/G$ is rational, and its function field is $K(x)$.
Hence the automorphism group $\bar H\cong H$ of $\cN_{q,r}/G$ induced by $H$ has exactly two fixed places and acts semiregularly elsewhere; see \cite[Hauptsatz 8.27]{Hup}.
Since $C\leq H$, the two places fixed by $\bar H$ are the place $\bar {\mathcal P}_\infty$ under $\mathcal P_\infty$ and the zero $\bar P_0$ of $x$.
Let $\Omega=\{P_{(0,0)}, P_{(0,a_2)}, \ldots, P_{(0,a_{q^{r-1}})}\}$ be the orbit of $G$ lying over $\bar P_0$, so that $Aut_K(\cN_{q,r})$ acts on $\Omega$; we denote by $P_{(0,0)}\in\Omega$ the zero of $y$, centered at the origin $(0,0)$.
The group $H$ has a fixed point in $\Omega$ by the Orbit-Stabilizer theorem, and $P_{(0,0)}$ is the only fixed place of $C$ other than $\mathcal P_{\infty}$; thus, $H$ fixes $P_{(0,0)}$.

Therefore, $H$ fixes the unique pole of $x$ and $y$, fixes the unique zero of $y$, and acts on the $q^{r-1}$ simple zeros of $x$. This implies that a generator $h$ of $H$ acts as $h(x)=\mu x$, $h(y)=\rho y$ for some $\mu,\rho\in K^*$.
By direct computation, $h$ is an automorphism of $\cN_{q,r}$ if and only if $\rho=\rho^q$ and $\mu^{\frac{q^r-1}{q-1}}=\rho$. Hence, $H=C$. 
\end{proof}

The following result generalizes Theorem \ref{AutNormTrace}.

\begin{theorem}\label{AutMonom}
Suppose that $m\not\equiv1\pmod{p^n}$ and $B(X)$ has just one root in $K$, so that Equation \eqref{EqSeparated} reads
$$b_m\left(X+\frac{b_{m-1}}{m b_m}\right)^m = A(Y).$$
Then one of the following two cases occurs.
\begin{itemize}
\item[(i)] $m$ divides $p^n+1$ and $A(Y)$ is $p^n$-linearized, that is, $A(Y)=a_n Y^{p^n}+a_0 Y$.
In this case, $\cC$ is projectively equivalent to the curve $\cQ_m$ with equation $X^m=Y^{p^n}+Y$ described in Case {\rm 1} of Lemma {\rm \ref{lem3}}.
\item[(ii)] $m$ does not divide $p^n+1$ or $A(Y)$ is not $p^n$-linearized.
Let $d=\gcd\left(j\geq1 : a_j\ne0\right)$ be the largest integer such that $A(Y)$ is $p^d$-linearized.
Then $Aut_K(\cC)$ has order $p^n m(p^d-1)$ and $Aut_K(\cC)=G\rtimes C$, where $G=\left\{(x,y)\mapsto(x,y+a)\mid A(a)=0\right\}$ and
$$ C=\left\{(x,y)\mapsto\left(bx+\frac{(b-1)b_{m-1}}{mb_m},b^m y\right)\mid b^{m(p^d-1)}=1\right\}. $$
\end{itemize}
\end{theorem}

\begin{proof}
Let $S$ be the stabilizer of $\mathcal P_\infty$ in $Aut_K(\cC)$.
By direct checking, $S$ contains the semidirect product $G\rtimes C$.
By Lemma \ref{lem2}, $S=G\rtimes H$, where $H$ is a cyclic group of order coprime to $p$. By Schur-Zassenhaus theorem, $H$ contains $C$ up to conjugation.
Arguing as in the proof of Theorem \ref{AutNormTrace}, $\cC/G$ is rational, and any nontrivial of the induced automorphism group $\bar H\cong H\leq Aut_K(\cC/G)$ fixes the pole $\bar{\mathcal P}_\infty$ of $x$ and the zero $\bar P$ of $x+\frac{b_{m-1}}{m b_m}$.
Hence $H$ acts on the $p^n$ distinct places of $\cC$ lying over $\bar P$, and $H$ fixes one of them by the Orbit-Stabilizer theorem. The only fixed place of $C$ different from $\mathcal P_{\infty}$ is the unique zero $P$ of $y$, centered at the affine point $(\frac{-b_{m-1}}{m b_m},0)$; thus, $H$ fixes $P$.
Let $h$ be a generator of $H$. We have shown that $h$ fixes the zero and the pole of $y$, which implies $h(y)=\rho y$ for some $\rho\in K$. Also, $h$ fixes the pole and acts on the simple zeros of $x+\frac{b_{m-1}}{m b_m}$; this implies $h(x+\frac{b_{m-1}}{m b_m}) = \mu (x+\frac{b_{m-1}}{m b_m})$ for some $\mu\in K$, that is, $h(x)=\mu x + \frac{(\mu-1)b_{m-1}}{m b_m}$.
By direct checking, $h$ normalizes $G$ if and only if $A(\mu a)=0$ for all $a\in K$ satisfying $A(a)=0$. As $A(Y)$ is separable, this happens if and only if $A(\mu Y)=A(Y)$. This is equivalent to $\mu\in\mathbb F_{p^d}^*$, with $d$ defined as in the statement of this theorem.
Then, in order for $h$ to be an automorphism of $\cC$, we have $\rho^m=\mu$.
We have shown that $S=G\rtimes C$.

From Lemma \ref{lem3}, either $Aut_K(\cC)=G\rtimes C$ and Case {\it (ii)} holds, or $\cC$ is isomorphic to the curve $\mathcal Q_s: X^s=Y^{\bar q}+Y$ with $s\mid(\bar q+1)$, $s<\bar q$.
Suppose that $\cC\cong \cQ_s$. By Lemma \ref{lem2} the Sylow $p$-subgroups of $Aut_K(\cC)$ and $Aut_K(\cQ_s)$ have size $p^n$ and $\bar q$ respectively, so that $\bar q=p^n$; as $g(\cC)=g(\cQ_s)$, we have $s=m$.
The normalizer in $Aut_K(\cQ_m)$ of a Sylow $p$-subgroup contains a cyclic group of order $p^n-1$, by Lemma \ref{lem3}(b). Hence, the same holds in $Aut_K(\cC)$ and $d=n$;
this means that $\cC$ has equation
\begin{equation}\label{EqTuttoLin}
b_m\left(X+\frac{b_{m-1}}{m b_m}\right)^m = a_n Y^{p^n} + a_0 Y.
\end{equation}
Conversely, if $\cC$ is defined by Equation \eqref{EqTuttoLin}, then $\cC$ is isomorphic to $\cQ_m$. In fact, define $\varphi:(x,y)\mapsto(x^\prime,y^\prime):=(\gamma x,\delta a_0 y)$ with $\delta^{p^n-1}=a b^{-p^n}$ and $\gamma^m=\delta$. Then $K(x,y)=K(x^\prime,y^\prime)$ and $\varphi(\cC)=\cQ_m$.
Now the proof is complete.
\end{proof}

Next result provides a converse to Theorem \ref{AutMonom} and extends \cite[Theorem 12.8]{HKT}.

\begin{theorem}\label{MonomAut}
Let $d=\gcd\left(j\geq1 : a_j\ne0\right)$ be the largest integer such that $A(Y)$ is $p^d$-linearized.
If $|Aut_K(\cC)_{P_\infty}|/|Aut_K(\cC)_{P_\infty}^{(1)}|\geq m(p^d-1)$, then $B(X)$ has a unique root in $K$, that is,
$$ B(X)=b_m\left(X+\frac{b_{m-1}}{mb_m}\right)^m. $$
\end{theorem}

\begin{proof}
Let $S$ be the stabilizer of $\mathcal P_\infty$ in $Aut_K(\cC)$, $H$ be a cyclic complement of $S^{(1)}$ in $S$, and $\alpha$ be a generator of $H$.
From Lemma \ref{lem2}, $G=\{(x,y)\mapsto(x,y+a)\mid A(a)=0\}$ is normal in $S$. Hence, $\alpha$ is an automorphism of the quotient curve $\cC/G$; by Lemma \ref{lem1}(e), $\cC/G$ is rational with function field $K(x)$.
From \cite[Haptsatz 8.27]{Hup}, $\alpha$ has two fixed places in $K(x)$ and acts semiregularly elsewhere. One of the two places is the pole of $x$, lying under $\mathcal{P}_\infty$; the other place is the zero of $x^\prime:=x+u$ for some $u\in K$. Thus $\alpha(x^\prime)=bx^\prime$, for some $b\in K^*$ of order $ord(b)=ord(\alpha)$.
Since $\alpha$ fixes the unique pole $\mathcal{P}_\infty$ of $y$ and the Weierstrass semigroup $H(\mathcal{P}_\infty)$ is generated by $-v_{\mathcal{P}_\infty}(x)=p^n$ and $-v_{\mathcal{P}_\infty}(x)=m$, we have that $\alpha(y)=ay+Q(x)$, where $a\in K^*$ and $Q(X)$ is a polynomial satisfying either $Q(X)=0$ or $\deg(Q(X))\cdot p^n<m$.

Let $B^\prime(X^\prime):=B(X)=B(X^\prime-u)$ and $Q^\prime(X^\prime):=Q(X)=Q(X^\prime-u)$.
Since $\alpha$ is an automorphism of $\cC$, the polynomial $A(aY+Q^\prime(X^\prime))-B^\prime(b X^\prime)$ is a multiple of the polynomial $A(Y)-B^\prime(X^\prime)$, say
\begin{equation}\label{Transf}
A(aY+Q^\prime(X^\prime))-B^\prime(b X^\prime) = k_1(A(Y)-B^\prime(X^\prime))
\end{equation}
with $k_1\in K^*$.
As $A$ is a separable polynomial, Equation \eqref{Transf} implies $A(aY)=kA(Y)$ and hence $k_1=a^{p^j}$ for any $j$ such that $a_j\ne0$; thus, $k_1=a$ and $a^{p^d-1}=1$. Equation \eqref{Transf} also implies $B^\prime(bX^\prime)=B^\prime(X^\prime)+A(Q^\prime(X^\prime))$ and hence $k_1=b^m$ from the comparison of monomials ${X^\prime}^m$; thus, $(b^m)^{p^d-1}=1$ which yields $|H|=m(p^d-1)$.

Let $\beta:=\alpha^{p^d-1}$, which has order $m$ acts as $\beta(x^\prime)=b^{p^d-1}x^\prime$, $\beta(y)=y+Q^\prime(b^{p^d-2}x^\prime)$.
As $\beta\in Aut_K(\cC)$, we have
$$
A(Y+Q^\prime(b^{p^d-2}X^\prime))-B^\prime(b^{p^d-1} X^\prime) = k_2(A(Y)-B^\prime(X^\prime))
$$
with $k_2\in K^*$. Then $k_2=1$ and
\begin{equation}\label{Transf2}
B^\prime(b^{p^d-1} X^\prime)=B^\prime(X^\prime)+A(Q^\prime(b^{p^d-2}X^\prime)).
\end{equation}
We want to show that $\beta(y)=y$.
Suppose by contradiction that $Q^\prime(b^{p^d-2}X^\prime)\ne0$.
If $Q^\prime(b^{p^d-2}X^\prime)$ is a nonzero constant, then the order of $\beta$ is a multiple of $p$, a contradiction to $ord(\beta)=m$.
If $\deg(Q^\prime(b^{p^d-2}X^\prime))>1$, then in Equation \eqref{Transf2} the right-hand side has a non-vanishing term of degree $p^n\cdot\deg(Q^\prime(X^\prime))$ while the left-hand side has not, a contradiction.
Therefore, $\beta(x^\prime)=b^{p^d-1}x^\prime$ and $\beta(y)=y$, with $ord(b)=m(p^d-1)$. Since $\beta$ is an automorphism of $\cC$, $B^\prime(X^\prime)=\lambda {X^\prime}^m$ for some $\lambda\in K^*$, that is, $B(X)=b_m\left(X+\frac{b_{m-1}}{mb_m}\right)^m$. 
\end{proof}

Even if $B(X)$ is not a monomial, the argument of the proof of Theorem \ref{MonomAut} shows the following result.

\begin{proposition}\label{Condiz}
Let $Aut_K(\cC)_{P_\infty}=Aut_K(\cC)_{P_\infty}^{(1)}\rtimes H$ with $H=\langle\alpha\rangle$, and let $d=\gcd(j\geq1:a_j\ne0)$ be the largest integer such that $A(Y)$ is $p^d$-linearized.
Then $\alpha(x)=bx+c$ for some $b,c\in K$, and $\alpha(B(x))=a B(x)$ for some $a\in\mathbb F_{p^d}^*$.
\end{proposition}

\begin{remark}
Once that $B(X)$ is explicitely given, Proposition {\rm \ref{Condiz}} provides a method to find $H$.
In fact, $H$ has one fixed affine place in $K(x)$ and acts semiregularly on the other affine places; also, $H$ acts on the zeros of $B(x)$ with the same multiplicity.
For instance:
\begin{itemize}
\item If $B(X)$ has more than one root, but only one root with fixed multiplicity $M>1$, then $|H|$ divides either $M$ or $M-1$.
\item If $B(X)$ has more than one root, and all the root have the same multiplicity $M>1$, then $H$ is trivial and $Aut_K(\cC)$ is a $p$-group of order $p^n$.
\end{itemize}
\end{remark}

\section{Multi point AG codes on the norm-trace curves}


Let $\ell,r\in\mathbb N$ with $r\geq3$, and let $\cN_{q,r}$ be the norm-trace curve as defined in Section \ref{Sec:Aut}. Let $\Omega=\{P_{(0,y_1)},\ldots,P_{(0,y_{q^{r-1}})}\}$ be the set of the $q^{r-1}$ $\mathbb F_{q^r}$-rational places of $\cN_{q,r}$ which are the zeros of $x$; the place $P_{(a,b)}$ is centered at the affine point $(a,b)$ of $\cN_{q,r}$. Let $\Theta:=\cN_{q,r}(\mathbb F_{q^r})\setminus\Omega$; note that $\Theta$ contains the place at infinity $P_\infty$.
As pointed out in the proof of Theorem \ref{AutNormTrace}, the principal divisors of the coordinate functions are the following:
\begin{itemize}
\item $(x)=\sum_{P\in \Omega}P - q^{r-1} P_\infty$ ;
\item $(y)=\frac{q^r-1}{q-1} P_{(0,0)} - \frac{q^r-1}{q-1} P_\infty$ .
\end{itemize}
Define the $\mathbb F_{q^r}$-divisors
$$G:=\sum_{P\in\Omega}\ell P\quad\textrm{and}\quad D:=\sum_{P\in\Theta}P.$$
Since $|\cN_{q,r}(\mathbb F_{q^r})|=q^{2r-1}+1$ (see \cite[Lemma 2]{G2003}), $G$ and $D$ have degree $\ell q^{r-1}$ and $q^{2r-1}+1-q^{r-1}$, respectively.
Denote by $C:=C_{\mathcal L}(D,G)$ the associated functional AG code over $\mathbb F_{q^r}$ having length $n= q^{2r-1}+1-q^{r-1}$, dimension $k$, and minimum distance $d$. The designed minimum distance is
$$ d^*=n-\deg (G) = q^{2r-1}+1-(\ell+1)q^{r-1}. $$
The designed minimum distance is attained by $C$.

\begin{proposition}
Whenever $d^*>0$, $C$ attains the designed minimum distance $d^*$.
\end{proposition}

\begin{proof}
By direct computation, the assumption $d^*>0$ is equivalent to $\ell<q^r$.
Take $\ell$ distinct elements $c_1,\ldots,c_\ell\in \mathbb F_{q^r}^*$ and let
$$f:=\prod_{i=1}^{\ell} \left(\frac{x-c_i}{x}\right).$$
The pole divisor of $f$ is exactly $G$, so that $f\in\cL(G)$. By the properties of the norm and trace maps, $f$ has exactly $\ell q^{r-1}$ distinct $\mathbb F_{q^r}$-rational zeros. Thus, the weigth of $e_D(f)$ is $n-\ell q^{r-1}=d^*$.
\end{proof}

We compute the dimension of $C$.

\begin{proposition}
If $\frac{q^r-1}{q-1}-2\leq\ell\leq q^r-1$, then
$$k=\ell q^{r-1}+1-\frac{1}{2}\left(\frac{q^r-1}{q-1}-1\right)\left(q^{r-1}-1\right).$$
\end{proposition}

\begin{proof}
Since $n>\deg(G)>2g-2$, $k=\deg(G)+1-g$ by the Riemann-Roch Theorem.
\end{proof}

\begin{proposition} \label{monomiallyeq}
The code $C$ is monomially equivalent to the extended one-point code $C_{ext}(D,G^\prime)$, where $G^\prime=\ell q^{r-1}\mathcal{P}_{\infty}$.
\end{proposition}

\begin{proof}
We have $G=G^\prime+(x^\ell)$ and hence $\cL(G^\prime)=\{f\cdot x^\ell \mid f\in\cL(G)\}$. The codeword of $C_{\cL}(D,G^\prime)$ associated to $f\cdot x^\ell$ is obtained as
$$ \big( (f x^\ell)(P_1),\ldots,(f x^\ell)(\mathcal{P}_{\infty}),\ldots,(f x^\ell)(P_n) \big) = \big( f(P_1),\ldots,f(\mathcal{P}_{\infty}),\ldots,f(P_n) \big) \cdot M, $$
where $M$ is the diagonal matrix with diagonal entries $x(P_1)^\ell,\ldots,(t^{\ell q^{r-1}} x)(\mathcal{P}_{\infty})^\ell,\ldots,x(P_n)^\ell \in\mathbb F_{q^r}$, with $t$ a local parameter at $\mathcal{P}_\infty$. This means that $M$ defines a monomial equivalence between $C$ and $C_{ext}(D,G^\prime)$.
\end{proof}

The Weierstrass semigroup $H(P_\infty)$ at $P_\infty$ is known to be generated by $q^{r-1}$ and $\frac{q^r-1}{q-1}$; see \cite{BR2013}. Thus, Proposition \ref{monomiallyeq} allows us to compute the dimension of $C$ also in those cases for which the Riemann-Roch Theorem does not give a complete answer. 

\begin{corollary}
If $1\leq\ell\leq\frac{q^r-1}{q-1}-3$, then the dimension of $C$ is
$$ k=\ell+1+\frac{(q-1)}{2}\bigg \lfloor \frac{\ell}{q} \bigg \rfloor \bigg (\bigg \lfloor \frac{\ell}{q} \bigg \rfloor+1 \bigg)+\frac{(q^2-3q+2)}{2}+\Delta,$$
where,
$$\Delta= \frac{(q-1)^2}{2} \bigg(\frac{\ell}{q}-1 \bigg)^2 + \bigg( \frac{(q-3)(q-1)}{2}\bigg) \bigg( \frac{\ell}{q}-1\bigg)+\frac{q(q-1)}{2} \bigg( \frac{\ell}{q}-1\bigg),$$
 if $\ell \equiv 0 \pmod q$; 
$$\Delta= \frac{(q-1)^2}{2} \bigg \lfloor \frac{\ell}{q}\bigg \rfloor^2+\bigg( \frac{(q-3)(q-1)}{2}\bigg)\bigg \lfloor \frac{\ell}{q} \bigg \rfloor +\frac{q(q-1)}{2} \bigg \lfloor \frac{\ell}{q} \bigg \rfloor,$$ 
 if  $q \equiv q-1 \pmod q$;
$$\Delta= \frac{(q-1)}{2} \bigg [ \bigg( \ell- \bigg \lfloor \frac{\ell}{q} \bigg \rfloor q \bigg) \bigg \lfloor \frac{\ell}{q} \bigg \rfloor^2+\bigg( q-\ell+\bigg \lfloor \frac{\ell}{q} \bigg \rfloor q-1\bigg) \bigg( \bigg \lfloor \frac{\ell}{q} \bigg \rfloor-1\bigg)^2\bigg ]+\bigg( \frac{q-3}{2}\bigg) \bigg[ \bigg(\ell-\bigg \lfloor \frac{\ell}{q} \bigg \rfloor q\bigg) \bigg \lfloor \frac{\ell}{q} \bigg \rfloor$$
$$+\bigg(q-\ell+\bigg \lfloor \frac{\ell}{q} \bigg \rfloor q -1\bigg) \bigg( \bigg \lfloor \frac{\ell}{q} \bigg \rfloor-1\bigg)\bigg]+\frac{1}{2} \bigg \lfloor \frac{\ell}{q} \bigg \rfloor \bigg( \ell - \bigg \lfloor \frac{\ell}{q} \bigg \rfloor q\bigg) \bigg( \ell - \bigg \lfloor \frac{\ell}{q} \bigg \rfloor q +1\bigg)$$
$$ + \frac{1}{2} \bigg ( \bigg \lfloor \frac{\ell}{q} \bigg \rfloor -1\bigg) \bigg(q-1-\ell+ \bigg \lfloor \frac{\ell}{q} \bigg \rfloor q\bigg) \bigg( q+\ell-\bigg \lfloor \frac{\ell}{q} \bigg \rfloor q\bigg),$$
 otherwise.
\end{corollary}

\begin{proof}
Let $c:=(q^r-1)/(q-1)$.
By the assumption on $\ell$, $\deg(G)<n$; hence, $k=\ell(G)$.
From Proposition \ref{monomiallyeq}, $k=\ell(G^\prime)$ with $G^\prime=\ell q^{r-1}\mathcal{P}_{\infty}$. This means that $k$ equals the number of non-gaps $h\in H(\mathcal{P}_\infty)$ at $\mathcal{P}_\infty$ satisfying $h\leq\ell q^{r-1}$. From \cite{G2003} (see also \cite{BR2013}), $k$ is the number of couples $(i,j)\in\mathbb N^2$ such that
$$ 0\leq i<q^r,\quad 0\leq j<q^{r-1},\quad i q^{r-1}+j c \leq \ell q^{r-1}. $$
Since $\ell\leq c-3$, this implies
$$ k=\sum_{i=0}^{\ell} \left(\left\lfloor\frac{(\ell-i)q^{r-1}}{c}\right\rfloor+1\right) =\ell+1+\sum_{s=0}^{\ell}\left\lfloor\frac{ s q^{r-1}}{c}\right\rfloor. $$
Write $s=aq+b$ with $a\geq0$ and $1\leq b\leq q$. The condition $s\leq\ell$ is equivalent to $a\leq \lfloor\frac{\ell-b}{q}\rfloor$ when $b<q$, and to $a\leq \lfloor\frac{\ell}{q}\rfloor-1$ when $b=q$. Hence, 
\begin{equation}\label{conto1}
 k =\ell+1+\sum_{a=0}^{\lfloor\frac{\ell}{q}\rfloor-1}\left\lfloor\frac{(aq+q)q^{r-1}}{c}\right\rfloor + \sum_{b=1}^{q-1} \sum_{a=0}^{\lfloor\frac{\ell-b}{q}\rfloor} \left\lfloor\frac{(aq+b)q^{r-1}}{c}\right\rfloor.
\end{equation}
By direct computation,
\begin{small}
\begin{equation}\label{conto2}
\sum_{a=0}^{\lfloor\frac{\ell}{q}\rfloor-1}\left\lfloor\frac{(aq+q)q^{r-1}}{c}\right\rfloor = \sum_{a=0}^{\lfloor\frac{\ell}{q}\rfloor-1}\left\lfloor(a+1)(q-1)+\frac{a+1}{c}\right\rfloor = \sum_{a=0}^{\lfloor\frac{\ell}{q}\rfloor-1}(a+1)(q-1) = \frac{1}{2}(q-1)\left\lfloor \frac{\ell}{q} \right\rfloor \left(\left\lfloor \frac{\ell}{q} \right\rfloor+1\right).
\end{equation}
\end{small}
Also,
$$ \frac{(aq+b)q^{r-1}}{c} = a(q-1)+b-1+ \frac{q^r-1+a(q-1)-b(q^{r-1}-1)}{q^r-1}. $$
Assume that $1\leq b\leq q-1$ and $0\leq a\leq \left\lfloor\frac{\ell-b}{q}\right\rfloor \leq \left\lfloor\frac{\ell}{q}\right\rfloor$. By the assumption on $\ell$ follows $a\leq q\frac{q^{r-2}-1}{q-1}$. Thus,
$$ \frac{q^r-1+a(q-1)-b(q^{r-1}-1)}{q^r-1}>0, \quad \frac{q^r-1+a(q-1)-b(q^{r-1}-1)}{q^r-1}<1, $$
so that $\left\lfloor\frac{(aq+b)q^{r-1}}{c}\right\rfloor=a(q-1)+b-1$.
Thus,
$$ \sum_{b=1}^{q-1} \sum_{a=0}^{\lfloor\frac{\ell-b}{q}\rfloor} \left\lfloor\frac{(aq+b)q^{r-1}}{c}\right\rfloor = \sum_{b=1}^{q-1} \sum_{a=0}^{\lfloor\frac{\ell-b}{q}\rfloor}\left(a(q-1)+b-1\right) =$$
$$\frac{(q-1)}{2} \sum_{b=1}^{q-1} \left\lfloor\frac{\ell-b}{q}\right\rfloor^2 + \bigg( \frac{q-3}{2} \bigg) \sum_{b=1}^{q-1} \left\lfloor\frac{\ell-b}{q}\right\rfloor + \sum_{b=1}^{q-1} b \left\lfloor\frac{\ell-b}{q}\right\rfloor + \frac{q^2-3q+2}{2}.$$
Denote by,
$$A=\frac{(q-1)}{2} \sum_{b=1}^{q-1} \left\lfloor\frac{\ell-b}{q}\right\rfloor^2 , \quad B=\bigg( \frac{q-3}{2} \bigg) \sum_{b=1}^{q-1} \left\lfloor\frac{\ell-b}{q}\right\rfloor, \quad C= \sum_{b=1}^{q-1} b \left\lfloor\frac{\ell-b}{q}\right\rfloor .$$
We note that for a given $b=1,\ldots,q-1$, holds that $\bigg \lfloor \frac{\ell-b}{q} \bigg \rfloor \ne \bigg \lfloor \frac{\ell-b-1}{q} \bigg \rfloor$ if and only if $\ell-b \equiv 0 \pmod q$.
 Thus if $\ell \equiv 0 \pmod q$ then $\bigg \lfloor \frac{\ell-b}{q} \bigg \rfloor=\frac{\ell}{q} - \bigg \lceil \frac{b}{q} \bigg \rceil=\frac{\ell}{q}-1$, for every $b=1,\ldots,q-1$; if $\ell \equiv q-1 \pmod q$ then $\bigg \lfloor \frac{\ell-b}{q} \bigg \rfloor=\bigg \lfloor \frac{\ell}{q} \bigg \rfloor$; while $\bigg \lfloor \frac{\ell-b}{q} \bigg \rfloor=\bigg \lfloor \frac{\ell}{q} \bigg \rfloor$ for $b=1,\ldots,\ell-\bigg \lfloor \frac{\ell}{q} \bigg \rfloor q$ and $\bigg \lfloor \frac{\ell-b}{q} \bigg \rfloor=\bigg( \bigg \lfloor \frac{\ell}{q} \bigg \rfloor-1\bigg)$ for $b=\ell-\bigg \lfloor \frac{\ell}{q} \bigg \rfloor q+1,\ldots,q-1$, if $\ell \not\equiv 0,q-1 \pmod q$. In particular this implies that

$$A= \frac{(q-1)}{2}\sum_{b=1}^{q-1}  \bigg( \frac{\ell}{q}-1\bigg)^2= \frac{(q-1)^2}{2} \bigg( \frac{\ell}{q}-1\bigg)^2,$$
 if  $\ell \equiv 0 \pmod q$, 
$$A=\frac{(q-1)}{2} \sum_{b=1}^{q-1} \bigg \lfloor \frac{\ell}{q} \bigg \rfloor^2=\frac{(q-1)^2}{2} \bigg \lfloor \frac{\ell}{q} \bigg \rfloor^2,$$
if $\ell \equiv q-1 \pmod q$, and
$$A=\frac{(q-1)}{2}\sum_{b=1}^{\ell-\big \lfloor \frac{\ell}{q} \big \rfloor q} \bigg \lfloor \frac{\ell}{q} \bigg \rfloor^2+ \frac{(q-1)}{2} \sum_{b=\ell-\big \lfloor \frac{\ell}{q} \big \rfloor q+1}^{q-1} \bigg( \bigg \lfloor \frac{\ell}{q} \bigg \rfloor-1\bigg)^2=$$
$$\frac{(q-1)}{2} \bigg[  \bigg( \ell- \bigg \lfloor \frac{\ell}{q} \bigg \rfloor q \bigg) \bigg \lfloor \frac{\ell}{q} \bigg \rfloor^2+\bigg( q-\ell+\bigg \lfloor \frac{\ell}{q} \bigg \rfloor q-1\bigg) \bigg( \bigg \lfloor \frac{\ell}{q} \bigg \rfloor-1\bigg)^2\bigg],$$
 otherwise. Analagously,
$$B= \frac{(q-3)}{2}\sum_{b=1}^{q-1}  \bigg( \frac{\ell}{q}-1\bigg)= \frac{(q-3)(q-1)}{2} \bigg( \frac{\ell}{q}-1\bigg),$$
if $\ell \equiv 0 \pmod q$,
$$B= \frac{(q-3)}{2} \sum_{b=1}^{q-1} \bigg \lfloor \frac{\ell}{q} \bigg \rfloor=\frac{(q-1)(q-3)}{2} \bigg \lfloor \frac{\ell}{q} \bigg \rfloor,$$
 if $\ell \equiv q-1 \pmod q$, while 
$$B=\frac{(q-3)}{2}\sum_{b=1}^{\ell-\big \lfloor \frac{\ell}{q} \big \rfloor q} \bigg \lfloor \frac{\ell}{q} \bigg \rfloor+ \frac{(q-3)}{2} \sum_{b=\ell-\big \lfloor \frac{\ell}{q} \big \rfloor q+1}^{q-1} \bigg( \bigg \lfloor \frac{\ell}{q} \bigg \rfloor-1\bigg)=$$
$$\bigg( \frac{q-3}{2}\bigg) \bigg[ \bigg(\ell-\bigg \lfloor \frac{\ell}{q} \bigg \rfloor q\bigg) \bigg \lfloor \frac{\ell}{q} \bigg \rfloor+\bigg(q-\ell+\bigg \lfloor \frac{\ell}{q} \bigg \rfloor q -1\bigg) \bigg( \bigg \lfloor \frac{\ell}{q} \bigg \rfloor-1\bigg)\bigg]$$
otherwise, and
$$C=\sum_{b=1}^{q-1} b \bigg( \frac{\ell}{q}-1\bigg)=\frac{q(q-1)}{2}\bigg( \frac{\ell}{q}-1\bigg),$$
if $\ell \equiv 0 \pmod q$,
$$C=\sum_{b=1}^{q-1} b \bigg \lfloor \frac{\ell}{q} \bigg \rfloor=\frac{q(q-1)}{2} \bigg \lfloor \frac{\ell}{q} \bigg \rfloor,$$
if $\ell \equiv q-1 \pmod q$ and
$$C=\sum_{b=1}^{\ell-\big \lfloor \frac{\ell}{q} \big \rfloor q} b \bigg \lfloor \frac{\ell}{q} \bigg \rfloor+\sum_{b=\ell-\big \lfloor \frac{\ell}{q} \big \rfloor q+1}^{q-1} b \bigg( \bigg \lfloor \frac{\ell}{q} \bigg \rfloor-1\bigg)=$$
$$\frac{1}{2} \bigg \lfloor \frac{\ell}{q} \bigg \rfloor \bigg( \ell - \bigg \lfloor \frac{\ell}{q} \bigg \rfloor q\bigg) \bigg( \ell - \bigg \lfloor \frac{\ell}{q} \bigg \rfloor q +1\bigg)+ \frac{1}{2} \bigg ( \bigg \lfloor \frac{\ell}{q} \bigg \rfloor -1\bigg) \bigg(q-1-\ell+ \bigg \lfloor \frac{\ell}{q} \bigg \rfloor q\bigg) \bigg( q+\ell-\bigg \lfloor \frac{\ell}{q} \bigg \rfloor q\bigg),$$
otherwise. The claim now follows writing $k=\ell+1+\frac{(q-1)}{2}\bigg \lfloor \frac{\ell}{q} \bigg \rfloor \bigg (\bigg \lfloor \frac{\ell}{q} \bigg \rfloor+1 \bigg)+\frac{(q^2-3q+2)}{2}+A+B+C$.
\end{proof}

We show that the automorphism group of $\cN_{q,r}$ is inherited by the code $C$.

\begin{proposition}
The automorphism group of $C$ has a subgroup isomorphic to 
$$ (Aut_K(\cN_{q,r})\rtimes Aut_K(\mathbb F_{q^r}))\rtimes \mathbb F_{q^r}^*. $$
\end{proposition}

\begin{proof}
The group $Aut_K(\cN_{q,r})$ is defined over $\mathbb F_{q^r}$, so that $Aut_{\mathbb F_{q^r}}(\cN_{q,r})=Aut_K(\cN_{q,r})$.
The support $supp(G)$ of the divisor $G$ is an orbit of $Aut_K(\cN_{q,r})$, and $Aut_K(\cN_{q,r})$ acts on the support $supp(D)=\cN_{q,r}(\mathbb F_{q^r})\setminus supp(G)$ of the divisor $D$. Also, all places contained in $supp(G)$ have the same weight in $G$, which implies $\sigma(G)=G$ for any $\sigma\in Aut_K(\cN_{q,r})$; analogously, $\sigma(D)=D$.
Therefore, $Aut_{\mathbb F_{q^r},D,G}(\cN_{q,r})$ is isomorphic to $Aut_K(\cN_{q,r})$.

From the proof of Theorem \ref{AutNormTrace} follows that $Aut_K(\cN_{q,r})$ has just two short orbits on $\cN_{q,r}$.
Namely, one short orbit is the singleton $\{\mathcal{P}_\infty\}$, which is fixed by the whole group $Aut_K(\cN_{q,r})$; the other short orbit is the set $\Omega$ of the zeros of $x$, which has size $q^{r-1}$ and is fixed pointwise by the complement $H$ of the $p$-group $G$.
Hence, any non-trivial element $\sigma\in Aut_K(\cN_{q,r})$ is fixes at most $N:=q^{r-1}+1$ places on $\cN_{q,r}$.
Since the length $n$ of $C$ is bigger than $N$, the claim follows from Proposition \ref{tivoglioiniettivo}.
\end{proof}

\section*{Acknowledgement}
The first author would like to thank his supervisor, Prof. Massimiliano Sala.

\begin{flushleft}
Matteo Bonini\\
Dipartimento di Matematica,\\
University of Trento,\\
e-mail: {\sf matteo.bonini@unitn.it}
\end{flushleft}

\begin{flushleft}
Maria Montanucci\\
Dipartimento di Matematica, Informatica ed Economia,\\
University of Basilicata,\\
e-mail: {\sf maria.montanucci@unibas.it}
\end{flushleft}

\begin{flushleft}
Giovanni Zini\\
Dipartimento di Matematica e Informatica,\\
University of Florence,\\
e-mail: {\sf gzini@math.unifi.it}
\end{flushleft}

\end{document}